\documentclass[12points]{amsart}
\usepackage{amsfonts}
\usepackage{amssymb}
\usepackage{graphicx}
\usepackage{amsmath}

\newtheorem{theorem}{Theorem}[section]
\newtheorem{lemma}{Lemma}
\newtheorem{corollary}{Corollary}
\theoremstyle{definition}
\newtheorem{definition}{Definition}

\theoremstyle{remark}

\numberwithin{equation}{section}

\everymath{\displaystyle}

\begin{document}

\title{Simplicity of the Lie algebra of skew symmetric elements of a Leavitt path algebra}

%    Information for first author
\author{Adel Alahmedi}
%    Address of record for the research reported here
\address{Department of Mathematics, King Abdulaziz University, P.O.Box 80203, Jeddah, 21589, Saudi Arabia}
%    Current address
\curraddr{Department of Mathematics, King Abdulaziz University, P.O.Box 80203, Jeddah, 21589, Saudi Arabia}
\email{adelnife2@yahoo.com}

\author{Hamed Alsulami}
\address{Department of Mathematics, King Abdulaziz University, P.O.Box 80203, Jeddah, 21589, Saudi Arabia}
\email{hhaalsalmi@kau.edu.sa}


\begin{abstract}
 For a field $F$ of characteristic not $2$ and a directed row-finite graph $\Gamma$ let $L(\Gamma)$ be the Leavitt Path Algebra with standard involution $*.$ We study the lie algebra of $K=K(L(\Gamma),*)$ of  $*-$skew-symmetric elements and find necessary and sufficient conditions for the Lie algebra $[K,K]$ to be simple.
\end{abstract}
\maketitle

\section{Introduction.}

G. Abrams and G. Aranda Pino in [1] have shown that a Leavitt path algebra of a row finite graph $\Gamma(V,E)$ is simple if and only if (i) the set of vertices $V$ has no proper hereditary and saturated subsets (ii) every cycle in $\Gamma$ has an exit. In [2] G. Abrams and Z. Mesyan found necessary and sufficient conditions for the Lie algebra $[L(\Gamma),L(\Gamma)]$ to be simple under the assumption that  $L(\Gamma)$ is simple. They proved that if $V$ is infinite then $[L(\Gamma),L(\Gamma)]$ is simple. If $V$ is finite then $[L(\Gamma),L(\Gamma)]$ is simple if and only if $1_{L(\Gamma)}=\sum_{v\in V}v\notin [L(\Gamma),L(\Gamma)].$ In this paper we study the Lie algebra $K=K(L(\Gamma),*)$ of $*-$skew-symmetric elements and prove that $[K,K]$ is simple if and only if the graph $\Gamma$ is almost simple, see the Definition $4$ below.

\section{\protect\bigskip Definitions and Terminology}

A (directed) graph $\Gamma =(V,E,s,r)$ consists of two sets $V$ and $E$ that are respectively called vertices and edges, and two maps $s,$ $r:E\rightarrow V$.The vertices $s(e)$ and $r(e)$ are referred to as the source and the range of the edge $e$, respectively. The graph is called row-finite if for all vertices $v\in V$, $\text{card}(s^{-1}(v))<\infty .$ A vertex $v$ for which $(s^{-1}(v))=\emptyset$ is called a sink. A vertex $v$ such that $r^{-1}(v)=\emptyset$ is called a source. A path $p=e_{1}\cdots e_{n}$ in a graph $\Gamma $ is a sequence of edges $e_{1},\cdots ,e_{n}$ such that $r(e_{i})=s(e_{i+1})$ for $i=1,\cdots,n-1.$ In this case we say that the path $p$ starts at the vertex $s(e_{1})$ and ends at the vertex $r(e_{n}).$ If $s(e_{1})=$ $r(e_{n}),$ then the path is closed. If $p=e_{1}\cdots e_{n}$ is a closed path and if the vertices $s(e_{1}),\cdots,s(e_{n})$ are distinct, then we call the path $p$ a cycle. A cycle of length $1$ is called a loop. Denote $V(p)=\{s(e_1),\cdots,s(e_n)\},$  $E(p)=\{e_1,\cdots, e_n\}.$
An edge $e\in E$ is called an exit of a cycle $C$ if $s(e)\in V(C),$ but $e\notin E(C).$

\bigskip

Let $X,Y$ are nonempty subsets of the set $V$ then we denote $E(X,Y)=\{e\in E \mid\, s(e)\in X ,  r(e)\in Y\}.$

\begin{definition}\label{def1}
\textit{We  call an edge $e\in E$ a \textbf{fiber} if $s(e)$ is source, $r(e)$ is  sink and $E(V, r(e))=\{e\}.$}
\end{definition}

\begin{definition}\label{def2}
\textit{We  call a connected graph $\Gamma$ a \textbf{fork} if , $\text{card}(V)>1,$ one vertex in $V$ is a source, whereas all  other vertices are sinks.}
\end{definition}

\begin{definition}\label{def3}
\textit{We call a vertex $v$ in a connected graph $\Gamma(V,E)$ a \textbf{balloon} over a nonempty subset $W$ of $V$ if (i) $v\notin W,$ (ii) there is a loop $C\in E(v,v),$ (iii) $E(v,W)\neq\emptyset,$ (iv) $E(v,V)=\{C\}\cup E(v,W),$ and (v) $E(V,v)=\{C\}.$}
\end{definition}

Let $\Gamma $ be a row-finite graph and let $F$ be a field. The Leavitt path $F$-algebra $L(\Gamma )$ is the $F$-algebra presented by the set of generators $\{v \mid\, v\in V\},$ $\{e,$ $e^{\ast }|$ $e\in E\}$ and the set of relators (1) $v_{i}v_{j}=\delta _{v_{i},v_{j}}v_{i}$ for all $v_{i},v_{j}\in V;$ (2) $s(e)e=er(e)=e,$ $r(e)e^{\ast }=e^{\ast }s(e)=e^{\ast}$ for all $e\in E;$ (3) $e^{\ast }f=\delta _{e,f}r(e),$ for all $e,$ $f\in E;$ (4) $v=\sum_{s(e)=v}ee^{\ast }$, for an arbitrary vertex $v$ which is not a sink.\newline  The mapping which sends $v$ to $v$ for $v\in V,$ $e$ to $e^{\ast }$ and $e^{\ast }$ to $e$ for $e\in E,$ extends to an involution of the algebra $L(\Gamma ).$ If $p=e_{1}\cdots e_{n}$ is a path, then
$p^{\ast}=e_{n}^{\ast }\cdots e_{1}^{\ast }.$\newline In what follows we consider only row-finite directed graphs.
\newline We will repeatedly consider the following graphs :

\includegraphics[width=.1\textwidth]{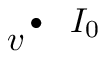}, a vertex ,\hspace{.5cm} \includegraphics[width=.1\textwidth]{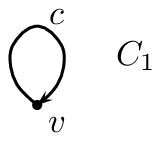}, a loop,  \hspace{.5cm}  and \includegraphics[width=.3\textwidth]{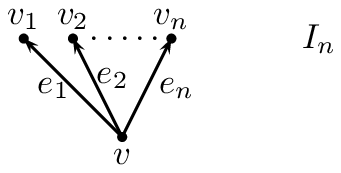}, a fork.

\bigskip

We say a vertex $v_2$ is a descendant of a vertex $v_1$ if there exists a path $p$ such that $s(p)=v_1,\, r(p)=v_2.$ \newline A nonempty subset $W\subseteq V$ is said to be hereditary if for an arbitrary element $w\in W$ all descendants of $w$ lie in $W$ (see[1]). If $W$ is a hereditary subset of $V$ then the ideal $I$ generated by $W$ in $L(\Gamma)$ is spanned by all elements $pq^*,$ where $p,q$ are pathes, $r(p)=r(q)\in W.$
\newline A nonempty subset $W\subseteq V$ is saturated if every non-sink $v\in V$ such that $r(s^{-1}(v))\subseteq W$ lies in $W$ (see [1]).
\newline Let $W$ be a hereditary saturated subset of $V,$ let $I$ be the ideal generated by $W$ in $L(\Gamma).$ Then the factor algebra $L(\Gamma)/I$ is isomorphic to the Leavitt path algebra $L(\Gamma')$ of
the quotient graph $\Gamma'=(V\setminus W, E\setminus E(V,W)),$ [6, Lemma 2.3]. \newline We view vertices $v\in V$ as paths of length $0,$ $s(v)=r(v)=v.$

\begin{definition}\label{def4}
\textit{We call a graph $\Gamma$ simple if the Leavitt path algebra $L(\Gamma)$ is simple.}
\end{definition}

In [1] it is shown that a graph $\Gamma=(V,E)$ is simple if and only if $V$ does not have proper hereditary saturated subsets and every cycle of $\Gamma$ has an exit.

\begin{definition}\label{def5}
\textit{We call a graph $\Gamma=(V,E)$ almost simple if $V$ contains subsets $V\supset V_1\supset V_2,$ such that the graph $\Gamma_2=(V_2, E(V_2,V_2))$ is simple, every vertex from $V_1\setminus V_2$ is a balloon over $V_2,$ every vertex $v\in V\setminus V_1$ is a source, $E(v,V)=E(v,V_1)$ consists of one edge which is a fiber.}
\end{definition}

Let $A$ be  an associative $F-$algebra. For elements $a,b\in A$, let $[a,b]=ab-ba$ be their  commutator. Then $A^{(-)}=(A, [,])$ is a Lie algebra.
The space of $*-$skew-symmetric elements $K=K(L(\Gamma),*)=\{a\in A\mid\, a^*=-a\}$ is a subalgebra of the Lie algebra $L(\Gamma)^{(-)}.$

\bigskip

The main result  of this paper is Theorem \ref{th1}. In section 3 we prove that if  the Lie algebra $[K,K]$ is simple then  the graph $\Gamma$ is almost simple.
In section 4 we prove the other direction using theorem of  I. Herstein which required that  the field $F$ is of characteristic not $2.$

\begin{theorem}\label{th1}
The Lie algebra $[K,K]$ is simple if and only if the graph $\Gamma$ is almost simple.
\end{theorem}

\section{ Almost Simple Graphs }

\begin{lemma}\label{le1}
Let $\Gamma(V,E)$ be a directed graph.
\begin{itemize}
  \item[(i)] If $e,f$ are different edges and $r(e)=s(f),$ then $[e-e^*,f-f^*]\neq 0.$
  \item[(ii)]  If $e,f$ are different edges and $r(e)=r(f),$ then $[e-e^*,f-f^*]\neq 0.$
\end{itemize}
\end{lemma}

\begin{proof}

\begin{itemize}
  \item[(i)] Since  $e,f$ are different edges we have  $e^*f=f^*e=0.$
  \newline Hence  $[e-e^*,f-f^*]=ef-ef^*+(fe)^*-fe+fe^*-(ef)^*.$ The nonzero elements among $ef,$ $ef^*,$ $(fe)^*,$ $fe,$ $fe^*,$ $(ef)^*$ are linearly independent: for example they can be included in a base of [4, Theorem 1]. Since the element $ef$ is nonzero, the assertion has been proved.
  \item[(ii)]Instead of the nonzero element $ef$ in the argument above we use the fact that $ef^*\neq0,$ $fe^*\neq0.$
\end{itemize}
\end{proof}

\begin{lemma}\label{le2}
If $[K,K]=0,$ then $\Gamma$ is a disjoint union of vertices, loops and forks.
\end{lemma}

\begin{proof}
This lemma immediately follows from Lemma~\ref{le1}.
\end{proof}

Recall that a nonempty subset $W\subseteq V$ is said to be hereditary if $r(s^{-1}(w))\subseteq W$ for every vertex $w\in W,$ see [1].
\newline A nonempty subset $W\subseteq V$ is saturated if every non-sink $v\in V,$ such that $r(s^{-1}(v))\subseteq W$ lies in $W,$ see [1].
\newline If $W$ is a hereditary and saturated subset of $V,$ then the ideal $I$ generated by $W$ in $L(\Gamma)$ is spanned by all elements $pq^*,$
where $p,q$ are paths, $r(p)=r(q)\in W.$ In this case $L(\Gamma)/I\cong L(\Gamma'),$ where $\Gamma'=(V\setminus W, E\setminus E(V,W)),$
see [3, Lemma 2.3].

\begin{lemma}\label{le3}
Let $\Gamma(V,E)$ be a connected graph and $e\in E$ be a fiber. Let $\Gamma'=(V\setminus\{r(e)\},E\setminus\{e\}).$
Then $L(\Gamma)\cong L(\Gamma')\oplus M_2(F).$
\end{lemma}

\begin{proof}
If $s^{-1}(s(e))=\{e\}$ then \includegraphics[width=.18\textwidth]{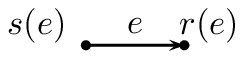} is an isolated subgraph, hence it is equal to $\Gamma,$ which contradicts our assumption. Hence $s^{-1}(s(e))\supsetneqq \{e\}.$ This implies that  $\{r(e)\}$ is a hereditary and saturated subset of $V.$
The ideal $I=id_{L(\Gamma)}(r(e))$ is the  $F-$span of $\{pq^*\mid\, r(p)=r(q)=r(e)\}.$ Looking at possible pathes we see that
$I=Fe+Fe^*+Fr(e)+Fee^*\cong M_2(F).$ Since every ideal with an identity is a direct summand we have $L(\Gamma)\cong L(\Gamma)/I\oplus M_2(F).$
Now $ L(\Gamma)/I\cong L(\Gamma')$ by [6 , Lemma 2.3(1)] as $\{r(e)\}$ is a hereditary saturated subset.
\end{proof}

\begin{corollary}\label{cr1}
Let $\Gamma$ be a connected graph, which is not a fork. Let $E_1=\{e\in E|\, e \text{ is a fiber }\},$
$\Gamma'=(V\setminus r(E_1), E\setminus E_1).$ Then $L(\Gamma)\cong L(\Gamma')\oplus A,$
where $A$ is a direct sum of $\text{card} (E_1)$ copies of $M_2(F).$ In particular, $[K(L(\Gamma),*),K(L(\Gamma),*)]\cong[K(L(\Gamma'),*),K(L(\Gamma'),*)].$
\end{corollary}

We call a graph $\Gamma$ fiber-less if none of its edges is a fiber.

\begin{lemma}\label{le4}
Let $\Gamma(V,E)$ be a connected fiber-less graph. Let $W$ be a nonempty hereditary and saturated  subset of $V.$ Let $I=id_{L(\Gamma)}(W)$ and let $K_{I}=K(I,*).$
If $[K,K]$ is simple then $[K_{I},K_{I}]\neq (0).$
\end{lemma}

\begin{proof}
Suppose that $[K_{I},K_{I}]=(0).$ Since $W,\, E(W,W)$ lie in $I$ and are $*$-invariant it follows from
Lemma~\ref{le2} that the graph $(W,E(W,W))$ is a disjoint union of graphs $I_0,$ $C_1,$ and $I_n.$

Suppose that the graph $(W, E(W,W))$ contains a loop $C_1.$ Then $E(v,V)=E(v,W)=\{c\}.$
If $e\in E(V,v)$ and $e\neq c,$ then both $e,c\in I$ and $[e-e^*,c-c^*]\neq0,$ by Lemma~\ref{le1}(ii), which contradicts our assumption. Hence $E(V,v)=\{c\}.$
Since the graph $\Gamma$ is assumed to be connected it follows that $\Gamma=C_1.$  But our assumption was that $[K,K]$ is a simple nonzero Lie algebra, a contradiction.

Let $I_n$ be a connected component of $(W, E(W,W)).$ Now, $E(v_i,V)=E(v_i,W)$ for $i=1,2,\ldots,n,$ so $v_1,\ldots,v_n$ are sinks.
If $e\in E(V,v_i)$ and $e\neq e_i,$ then $e,e_i\in I,$ $[e-e^*, e_i-e^*_i]\neq 0,$ a contradiction. Hence $E(V,v_i)=\{e_i\}.$ Similarly if $e\in E(V,v),$ then $e\in I,$
$[e-e^*,e_1-e^*_1]\neq0$ by Lemma~\ref{le1}(i). Hence $v$ is a source in $\Gamma$ and $E(v,V)=\{e_1,\ldots,e_n\}.$ Thus the subgraph $I_n$  is isolated in
$\Gamma,$ so $\Gamma=I_n$. Hence $L(\Gamma)\cong\oplus_{1}^{n} M_2(F)$ and hence $[K,K]=0.$ Contradiction.

Let  $I_0$  be a connected component of $(W,E(W,W)).$ Then $v$ is a sink in $\Gamma.$ Suppose that $v_1\in V\setminus W,$ and $e\in E$ with $s(e)=v_1$ and $r(e)=v.$

If $f\in E(V,v)$ and $e\neq f,$ then $e,f\in I$ and $[e-e^*,f-f^*]\neq 0$ by Lemma~\ref{le1}(ii), a contradiction. Hence $E(V,v)=\{e\}.$
\newline If $g\in E(V,v_1),$ then $ge\in I$ and $[ge-e^*g^*,e-e^*]=gee^*-ee^*g^*\neq0,$  a contradiction. Hence $v_1$ is a source. Thus $e$ is a fiber.  But we assumed that $\Gamma$ is fiber-less,
a contradiction. Thus $[K_{I},K_{I}]\neq (0).$
\end{proof}

\begin{lemma}\label{le5}
Let $\Gamma(V,E)$ be a fiber-less graph. Let $W$ be a nonempty hereditary and saturated  subset of $V.$ Let $I=id_{L(\Gamma)}(W)$ and let $K_{I}=K(I,*).$
Let $\Gamma'=(V\setminus W, E\setminus E(V\setminus W)).$ Let $K'=K(L(\Gamma'),*).$ If $[K,K]$ is simple,  then $[K',K']= (0).$
\end{lemma}

\begin{proof}
The Leavitt path algebra Since $L(\Gamma')$ of the quotient graph $\Gamma'$ is isomorphic to the factor algebra $L(\Gamma)/I$ (see [6, Lemma 2.3(1)].  By Lemma~\ref{le4} the commutator $[K_{I},K_{I}]$ is a nonzero ideal in $[K,K].$ Since $[K,K]$  is simple it follows that $[K,K]=[K_{I},K_{I}]\subseteq I.$
Hence  $[K',K']= (0).$
\end{proof}

\begin{corollary}
Let $\Gamma(V,E)$ be a fiber-less graph. Let $W$ be a nonempty hereditary and saturated  subset of $V.$ Let $I=id_{L(\Gamma)}(W)$ and let $K_{I}=K(I,*).$
Let $\Gamma'=(V\setminus W, E\setminus E(V\setminus W)).$
If $[K,K]$ is simple,  then $\Gamma'$ is a disjoint union of $I_0$, $C_1,$  and $I_n.$
\end{corollary}

\begin{lemma}\label{le6}
Let $\Gamma(V,E)$ be a fiber-less graph. Let $W$ be a nonempty hereditary and saturated  subset of $V.$
Let $\Gamma'=(V\setminus W, E\setminus E(V\setminus W)).$ If $[K,K]$ is  simple,
then $\Gamma'$ does not have connected components of the types  $I_0$ and $I_n.$
\end{lemma}

\begin{proof}
Let $I_0$  be a component of $\Gamma'.$ Then $E(v,V\setminus W)=\emptyset.$
If $s^{-1}(v)\neq\emptyset,$ then $r(s^{-1}(v))\subset W$ and since $W$ is saturated we will have $v\in W.$
Hence $v$ is a sink. Since $v$ is isolated in $\Gamma',$ it follows that  $E(V\setminus W, v)=\emptyset.$
Also since $v\notin W$ and $E(V\setminus W, v)=\emptyset,$ then $E(W,v)=\emptyset.$
Hence $E(V,v)=\emptyset$ and hence $v$ is isolated in $\Gamma,$ a contradiction.

Now suppose that $I_n$  is a component of $\Gamma'.$
Since $E(v_i,V\setminus W)=\emptyset,\, 1\leq i\leq n, $ it follows that $E(v_i, V)=E(v_i, W).$
Since $v_i\notin W$ we conclude that each $v_i$ is a sink. If $e$ is another edge different from $e_i$ arriving at $v_i,$ then  $s(e)$
must lie in  $W$ and hence $v_i$ belongs to $W,$  which is a contradiction. Now suppose that an edge $e$ arrives at $v.$
Again $s(e)$ can not be in $W$ and can not be in $V\setminus W.$ Hence $v$ is a source in $\Gamma.$
But then $e_1,\ldots,e_n$ are fibers. A contradiction.
\end{proof}

\begin{lemma}\label{le7}
Let $\Gamma(V,E)$ be a fiber-less graph. Let $W$ be a nonempty hereditary and saturated  subset of $V.$
Let $\Gamma'=(V\setminus W, E\setminus E(V\setminus W)).$ Suppose that $[K,K]$ is  simple, and $\Gamma'$ is a disjoint union of loops \includegraphics[width=.08\textwidth]{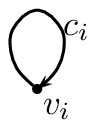}, indexed by elements of a set $\Omega,$ $i\in\Omega.$ Then $v_i,\,i\in\Omega$ are balloons over $W.$
\end{lemma}

\begin{proof}
Since $E(W,v_i)=\emptyset$ and $E(V\setminus W,v_i)=\{c_i\}$ as in $\Gamma'$ it follows that  $E(V,v_i)=\{c_i\}.$
Moreover, $E(v_i,V\setminus W)=\{c_i\}$ as in $\Gamma'.$ If $E(v_i,W)=\emptyset,$ then $v_i$ is isolated in $\Gamma,$
hence $E(v_i,W)\neq\emptyset.$ Thus $v_i$ is a balloon over $W.$
\end{proof}

\begin{lemma}\label{le8}
Let $d\geq 1$ be a positive integer. Then the Lie algebra\newline  $[K(M_d(F[t,t^{-1}]),*),K(M_d(F[t,t^{-1}]),*)]$ is not simple.
\end{lemma}

\begin{proof}
Note that the standard involution on $K(M_d(F[t,t^{-1}]))$ is $\left(f_{ij}(t)\right)^*=\left(f_{ji}(t^{-1})\right).$
For $d=1,$ $[K(F[t,t^{-1}]),K(F[t,t^{-1}])]=0.$ Let $d\geq 2.$ Let $J_n$ be the ideal of $F[t,t^{-1}]$ generated by $(1-t)^n.$
Since $((1-t)^n)^*=((1-t)^*)^n=(1-t^{-1})^n=(-1)^nt^{-n}(1-t)^n,$ it follows that $J^*=J.$ We will show that $[K(M_d(J_n),*),K(M_d(J_n),*)]\neq (0)$ for all $n\geq 1.$
It is sufficient to do it for $d=2.$ For $f,g\in F[t,t^{-1}],$ we have { \scriptsize{$$\left[\left(
                                                                                  \begin{array}{cc}
                                                                                    0 & f(t) \\
                                                                                    -f(t^{-1}) & 0 \\
                                                                                  \end{array}
                                                                                \right), \left(
                                                                                           \begin{array}{cc}
                                                                                             0 &  g(t) \\
                                                                                              -g(t^{-1}) & 0 \\
                                                                                           \end{array}
                                                                                         \right)\right]=\left(
                                                                                                          \begin{array}{cc}
                                                                                                            g(t)f(t^{-1})-f(t)g(t^{-1}) &0 \\
                                                                                                            0 & f(t)g(t^{-1})-f(t^{-1})g(t) \\
                                                                                                          \end{array}
                                                                                                        \right)$$ }}

Let $f(t)=(1-t)^n$ and $g(t)=(1-t)^m$ where $n<m.$ Then $f,g\in J_n$ and  the matrix on the right hand side is nonzero.
If  $[K(M_d(F[t,t^{-1}]),*),K(M_d(F[t,t^{-1}]),*)]$ is  simple, then it is equal to $[K(M_d(J_n),*),K(M_d(J_n),*)]$   and therefore lies in $M_d(J_n).$
But $\bigcap_{n} J_n=(0),$ a contradiction. \newline Hence   $[K(M_d(F[t,t^{-1}]),*),K(M_d(F[t,t^{-1}]),*)]$ is not simple.
\end{proof}

\begin{lemma}\label{le9}
Let $\Gamma(V,E)$ be a fiber-less graph. If $[K,K]$ is simple, then every cycle in $\Gamma$ has an exit or  is a loop.
\end{lemma}

\begin{proof}
Let $C$ be a cycle, which does not have an exit. Let $a=\sum_{i}v_i$ be the sum of all vertices on $C.$ Then $a$ is an idempotent and $aL(\Gamma)a\cong L(C).$
In [3, Prop. 3.5(iii)] it was shown that the Leavitt path algebra of a cycle with $d$ vertices is isomorphic to $M_d(F[t,t^{-1}]).$
For an arbitrary ideal $I\lhd L(C)$ denote $\widetilde{I} =id_{L(\Gamma)}(I).$
Then $a\widetilde{I}a\subseteq I.$ If $I^*=I,$ then $\widetilde{I}^*=\widetilde{I}.$ Let $J_n$ be the ideal of $F[t,t^{-1}]$ generated by $(1-t)^n.$
Let $I_n=id_{L(\Gamma)}(M_d(J_n)).$ As shown in the proof of Lemma \ref{le8}, $[K(I_n,*),K(I_n,*)]\neq (0).$ Since $[K,K]$ is simple it follows that $[K,K]=[K(I_n,*),K(I_n,*)].$
In particular, $[K,K]\cap L(C)\subseteq M_d(J_n).$ Since $\cap_n J_n=(0)$ we conclude that $[K,K]\cap L(C)=(0).$
By Lemma~\ref{le8} we have $(0)\neq [K(M_d(F[t,t^{-1}]),*),K(M_d(F[t,t^{-1}]),*)]\subseteq M_d(J_n) $ for $d\geq 2.$
Hence $d=1.$ Thus $C$ is a loop. If there is no edge arriving at $C$ from outside , then $C$ is isolated in $\Gamma$ and hence $\Gamma=C,$
which is a contradiction, since $[K,K]$ is simple. Hence there is an edge $e$ such that $s(e)\notin C$ and $r(e)=v\in V(C).$
Let $J_n$ be the ideal generated by $(v-C)^n$ in $L(C).$  Let $I_n=id_{L(\Gamma)}(J_n).$ Let $f,g\in J_n.$  Now, $ef,eg\in I_n.$
We have $[ef-(ef)^*,eg-(eg)^*]=-efg^*e^*-f^*g+egf^*e^*+gf^*\neq0.$ Now, $[K,K]=[K(I_n,*),K(I_n,*)]$ and $v[K,K]v\subseteq J_n.$
Since $\cap_n J_n=(0)$ it follows that $v[K,K]v=(0).$ But $v[ef-(ef)^*,eg-(eg)^*]v=g^*f-f^*g\neq 0$ for some $f,g\in L(C),$ a contradiction.
Hence $C$ has an exit.

\end{proof}

Let $\Gamma(V,E)$  be a fiber-less connected graph such that the Lie algebra \newline $[K(L(\Gamma),*),K(L(\Gamma),*)]$ is simple. Let $W_1, W_2$ be two nonempty hereditary saturated subsets of $V$ such that $W_1\cap W_2=\emptyset.$ If $w\in W_2,$ then by Lemma \ref{le7} the vertex $w$ is a balloon over $W_1.$ It means that there exist a loop \includegraphics[width=.08\textwidth]{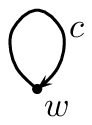}, $r^{-1}(w)=\{c\},$ all descendants of $w$ except $w$ itself lie in $W_2.$ Since $W_2$ is hereditary and $W_1\cap W_2=\emptyset$ it follows that the vertex $w$ is isolated in $\Gamma.$ Since $\Gamma$ is connected, $V=\{w\},$
a contradiction.  We showed that any two nonempty hereditary saturated subsets of $V$ intersect. Now let $\{W_i\}$ be the family of all nonempty hereditary saturated subsets of $V.$ We claim that $\cap_i W_i\neq\emptyset.$ Indeed, choose a vertex $v\in V.$  If $\cap_i W_i=\emptyset$ then there exists a subset $W_i$ such that $r(s^{-1}(v))\cap W_i=\emptyset.$ By Lemma \ref{le7} the vertex $v$ is a balloon over $W_i.$
Since $W_i$ does not contain any range of an edge originating at $v$ we conclude, as above, that the vertex $v$ is isolated in $V,$ a contradiction. Thus $W=\cap_i W_i $ is the smallest nonempty hereditary saturated subset of $V.$ Since every hereditary and  saturated subset of $W$ is a hereditary and  saturated subset of $V$ it follows that $W$ does not contains proper hereditary and  saturated subsets. Now, if $[K,K]$ is simple then by this and Lemma~\ref{le9} we see that $\Gamma_W=(W,E(W,W))$ satisfies the conditions of [1]. Hence $L(\Gamma_W)$ is a simple algebra. Hence we proved Theorem \ref{th1} in one direction: if the Lie algebra$[K,K]$ is simple then the graph $\Gamma$ is almost simple.

\section{ Simplicity of the Lie algebra $[K,K].$}

Now our aim is to prove that for an almost simple graph $\Gamma(V,E)$ the Lie algebra $[K(L(\Gamma),*), K(L(\Gamma),*)]$ is simple. \newline Recall that being almost simple means that $V\supset V_1\supset V_2,$ the graph $(V_2,E(V_2,V_2))$ is simple; every vertex from $V_1\setminus V_2$ is a balloon over $V_1;$ every vertex $v\in V\setminus V_1$ is a source, $E(v,V)=E(v,V_1)$ consists of one edge, which is a fiber.
\newline By Corollary \ref{cr1} of Lemma \ref{le3} we have $[K(L(\Gamma),*),K(L(\Gamma),*)]\cong[K(L(\Gamma_1),*),K(L(\Gamma_1),*)],$ where $\Gamma_1=(V_1, E(V_1,V_1)).$ \newline Therefore without loss of generality we will assume that the set of vertices $V$ of our row finite graph $\Gamma(V,E)$ contains a nonempty subset $W,$ such that the Leavitt path algebra $L(W,E(W,W))$ is simple and each vertex from $V\setminus W$ is a balloon over $ W.$
Our aim is to show that the Lie algebra $[K(L(\Gamma),*), K(L(\Gamma),*)]$ is simple.
\begin{center}
\includegraphics[width=.4\textwidth]{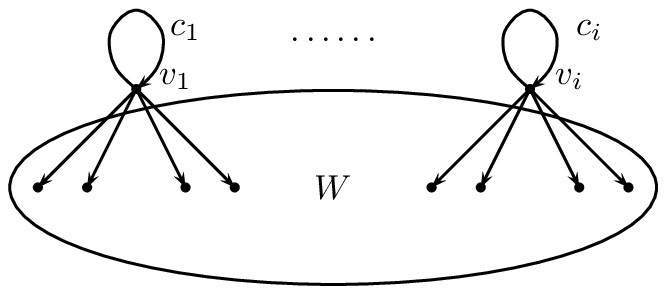}
\end{center}
The subset $W$ in $V$ is hereditary and  saturated. Let $I=id_{L(\Gamma)}(W).$ Then $L(\Gamma)/I$ is a direct sum of $\text{card} (V\setminus W)$- many copies of the algebra $F[t,t^{-1}].$ Denote $E_i=E(v_i,W),$ $v_i\in V\setminus W,$ $c_i$ is a loop from $E(v_i,v_i).$ Then
$I=L(W)+\sum_{i} c_{i}^{k_i}E_iL(W)+\sum_{j}L(W)E^*_{j}(c_{j}^*)^{r_j}+\sum_{i,j}c_{i}^{k_i}E_iL(W)E^*_{j}(c_{j}^*)^{r_j}$
This decomposition is related to the  family of pairwise orthogonal idempotents $\{u=\sum_{v\in W}v\}\bigcup^{.}(\cup_{i}^{.}\mathcal{E}_i),$
$\mathcal{E}_i=\{ c_i^kee^*(c_i^*)^k| k\geq 0, e\in E_i\}.$

\begin{lemma}\label{le10}
$I$ is a simple algebra.
\end{lemma}

\begin{proof}
We have $I=uIu+\sum_{i} uI\mathcal{E}_i+\sum_{i}\mathcal{E}_iIu+\sum_{i,j}\mathcal{E}_iI\mathcal{E}_j.$
Also, we have $I=L(W)+\sum_{i} c_{i}^{k_i}E_iL(W)+\sum_{j}L(W)E^*_{j}(c_{j}^*)^{r_j}+\sum_{i,j}c_{i}^{k_i}E_iL(W)E^*_{j}(c_{j}^*)^{r_j}.$ Let $J$ be a nonzero ideal in $I.$
Suppose that $\mathcal{E}_iJ\mathcal{E}_j\neq(0).$ Thus there exist $e\in\mathcal{E}_i, f\in\mathcal{E}_j,$ $k_i,k_j\geq 0$ such that
$(0)\neq c_i^{k_i}ee^*(c_i^*)^{k_i}Jc_j^{k_j}ff^*(c_j^*)^{k_j}.$   This subspace lies in $c_i^{k_i}eL(W)f^*(c_j^*)^{k_j}\cap J.$ Multiplying on the left by $e(c_i^*)^{k_i}\in I$ and on the right by $c_j^{k_j}f\in I,$ we  get $L(W)\cap J\neq 0.$  This and the simplicity of  $L(W)$ implies $L(W)\subseteq J$ and therefore $J=I.$
Suppose that $uJ\mathcal{E}_j\neq (0).$ Then there exist an edge $e\in E_j$ and $k\geq 0$ such that $(0)\neq uJc_j^kee^*(c_j^*)^k\subseteq L(W)e^*(c_j^*)^k\cap J.$ Multiplying on the right by $c_j^ke\in I,$ we get again $L(W)\cap J\neq (0)$ and hence $I=J.$ The case $\mathcal{E}_iJu\neq(0)$ is treated similarly. If $\mathcal{E}_iJu=uJ\mathcal{E}_j= \mathcal{E}_iJ\mathcal{E}_j=(0)$ for all $i,j$ then $J\subseteq L(W)$ and hence $I=J.$

\end{proof}

\begin{lemma}\label{le11}
Let $A$ be an arbitrary simple associative algebra with an involution $*:A\to A$ and three pairwise orthogonal symmetric idempotents $e_1,e_2,e_3.$ Then $K(A,*)=e_iK(A,*)e_i+[K(A,*),K(A,*)]$ for $i=1,2,3.$
\end{lemma}

\begin{proof}
Denote $K=K(A,*).$ For an arbitrary element $a\in A$ denote $\{a\}=a-a^*.$ For distinct $i,j,k\in\{1,2,3\}$ and for arbitrary elements $a,b\in A$ we have $\{e_iae_jbe_k\}=[e_iae_j-e_ja^*e_i, e_jbe_k-e_kb^*e_j]\in [K,K].$ Since $A=Ae_jA$ it follows that $\{e_iAe_k\}=\{e_iAe_jAe_k\}\subseteq [K,K].$
For arbitrary elements $a,b\in A$ we have $[ae_1-e_1a^*, e_1b-b^*e_1]=\{ae_1b\}-\{ae_1b^*e_1\}-\{e_1a^*e_1b\}+\{e_1a^*b^*e_1\}\in [K,K].$
Hence $\{ae_1b\}\in \{e_1A\}+[K,K].$ Since $A=Ae_1A$ it implies that $K=\{e_1A\}+[K,K].$ Again, for arbitrary elements $a,b\in A$ we have
$[e_1ae_2-e_2a^*e_1,e_2b-b^*e_2]=\{e_1ae_2b\}-\{e_1ae_2b^*e_2\}+\{e_2a^*e_1b^*e_2\}\in[K,K].$ Hence $\{e_1ae_2b\}\in \{e_1Ae_2\}+e_2Ke_2+[K,K]=e_2Ke_2+[K,K].$
Since $A=Ae_2A$ we get $\{e_1A\}\subseteq e_2Ke_2+[K,K]$ and in view of the earlier inclusions, $K=e_2Ke_2+[K,K].$ Similarly we can show that $K=e_iKe_i+[K,K]$ for $i=1,2,3.$
\end{proof}

\begin{lemma}\label{le12}
$[K(L(\Gamma),*),K(L(\Gamma),*)]=[K(I,*),K(I,*)]$. In particular, this Lie algebra is simple.
\end{lemma}

\begin{proof}
As above Let $V\setminus W=\{v_i\mid\, i\in\Omega\},\, E(v_i,v_i)=\{c_i\}.$
Then $K(L(\Gamma),*)=span\{c_i^p-(c_i^*)^p\mid\, i\in\Omega, p\geq 0\}+K(I,*).$ The algebra $I$ has an infinite family of pairwise orthogonal idempotents, which includes vertices from $W.$ Choose $w\in W.$ Then $K(I,*)=wK(I,*)w+[K(I,*),K(I,*)].$ \newline Now, $[c_i^p-(c_i^*)^p,K(I,*)]=[c_i^p-(c_i^*)^p,wK(I,*)w+[K(I,*),K(I,*)]]=[c_i^p-(c_i^*)^p,[K(I,*),K(I,*)]]$ which lies in $ [K(I,*),K(I,*)].$
Hence $[K(L(\Gamma),*),K(I,*)]\subseteq [K(I,*),K(I,*)].$ Let us analyze  an element $[c_i^p-(c_i^*)^p,c_i^q-(c_i^*)^q]$ which is equal to $-[c_i^p+(c_i^*)^p,c_i^q+(c_i^*)^q].$
Furthermore, $c_i^p+(c_i^*)^p=(c_i+c_i^*)^p+\sum_{k<p}\alpha_k(c_i^k+(c_i^*)^k)+h,\quad h\in H(I,*)=\{a\in I \mid\, a^*=a\}.$ This implies that $[c_i^p-(c_i^*)^p,c_i^q-(c_i^*)^q]\in [H(L(\Gamma),*),H(I,*)].$
I. Herstein proved that for a simple algebra $A$ of characteristic not equal to $2$ and of dimension not equal to $1$ or $4$ we have  $H(A,*)=span\{k^2\mid\, k\in K(A,*)\}.$ The algebra $I$ has an infinite family of pairwise orthogonal idempotents and therefore is infinite dimensional. If $h\in H(L(\Gamma),*), k\in K(I,*)$ then $[h,k^2]=[hk+kh,k]\in[K(L(\Gamma),*),K(I,*)]=[K(I,*),K(I,*)].$
Thus $[K(L(\Gamma),*),K(L(\Gamma),*)]=[K(I,*),K(I,*)].$ The latter algebra is simple, again by the theorem of I. Herstein.
\end{proof}

\section*{Acknowledgement}

The authors would like to thank professor Efim Zelmanov for his constant advise and valuable help during the preparation of this work.
The authors would also like to express their appreciation to professor  S. K. Jain for carefully reading the manuscript and for offering his comments.
\newline The authors sincerely thank the referee for numerous helpful remarks. This paper was funded by King Abdulaziz University, under grant No. (7-130/1433 HiCi).
The authors, therefore, acknowledge technical and financial support of KAU.

\bibliographystyle{amsplain}

\end{document}